\documentclass[11pt]{amsart}
\usepackage[utf8]{inputenc}
\usepackage{graphicx,xcolor,verbatim}
\usepackage{amscd,amsmath,amsfonts,amssymb}
\usepackage{mathtools}
\usepackage[foot]{amsaddr}
\usepackage[top=1in, bottom=1.25in, left=1in, right=1in]{geometry}
\usepackage[pdfborder=000,pdftex=true]{hyperref}

\newtheorem{theorem}{Theorem}
\newtheorem{lemma}[theorem]{Lemma}
\newtheorem{proposition}[theorem]{Proposition}
\newtheorem{remark}[theorem]{Remark}

\newtheorem{definition}{Definition}

\numberwithin{equation}{section}

\newcommand{\R}{\mathbb R}
\newcommand{\N}{\mathbb N}

\newcommand{\SD}{\Sigma_{\mathcal D}}

\newcommand{\SN}{\Sigma_{\mathcal N}}

\begin{document}
	
\title[Subcritical nonlocal problems with mixed boundary conditions]{Subcritical nonlocal problems with\\ mixed boundary conditions}

\thanks{A. Ortega is partially supported by the State Research Agency of Spain, under research project PID2019-106122GB-I00.}

\author{Giovanni Molica Bisci}
\address[G. Molica Bisci]{Dipartimento di Scienze Pure e Applicate (DiSPeA), Università degli Studi di Urbino Carlo Bo\\
Piazza della Repubblica, 13 - 61029 Urbino, Italy}
\email{\tt giovanni.molicabisci@uniurb.it}

\author{Alejandro Ortega}
\address[A. Ortega]{Departamento de Matem\'{a}ticas, Universidad Carlos III de Madrid\\
Av. Universidad 30, 28911 Legan\'{e}s (Madrid), Spain}
\email{\tt alortega@math.uc3m.es}

\author{Luca Vilasi}
\address[L. Vilasi]{Department of Mathematical and Computer Sciences, Physical Sciences and Earth Sciences\\
University of Messina\\
Viale F. Stagno d’Alcontres, 31 - 98166 Messina, Italy}
\email{\tt lvilasi@unime.it}

\keywords{Fractional Laplacian, Variational Methods, $\nabla$-Theorems, Mixed Boundary Data, Superlinear and Subcritical Nonlinearities.\\
\phantom{aa} 2010 AMS Subject Classification: Primary: 35J20, 35A15, 35S15; Secondary: 49J35, 35J61.
}
	
\begin{abstract}
In this paper, by variational and topological arguments based on
linking and $\nabla$-theorems, we prove the existence of multiple solutions for the following nonlocal problem with mixed Dirichlet-Neumann boundary data,
\begin{equation*}
        \left\{
        \begin{tabular}{lcl}
        $(-\Delta)^su=\lambda u+f(x,u)$ & &in $\Omega$, \\[2pt]
        $\mkern+51mu u=0$& &on $\Sigma_{\mathcal{D}}$, \\[2pt]
        $\mkern+36mu \displaystyle \frac{\partial u}{\partial \nu}=0$& &on $\Sigma_{\mathcal{N}}$,
        \end{tabular}
        \right.
\end{equation*}
where $(-\Delta)^s$, $s\in (1/2,1)$, is the spectral fractional Laplacian operator, $\Omega\subset\mathbb{R}^N$, $N>2s$, is a smooth bounded domain, $\lambda>0$ is a real parameter, $\nu$ is the outward normal to $\partial\Omega$,  $\Sigma_{\mathcal{D}}$, $\Sigma_{\mathcal{N}}$ are smooth $(N-1)$--dimensional submanifolds of $\partial\Omega$ such that
$\Sigma_{\mathcal{D}}\cup\Sigma_{\mathcal{N}}=\partial\Omega$,
$\Sigma_{\mathcal{D}}\cap\Sigma_{\mathcal{N}}=\emptyset$ and
$\Sigma_{\mathcal{D}}\cap\overline{\Sigma}_{\mathcal{N}}=\Gamma$ is a smooth $(N-2)$--dimensional submanifold of $\partial\Omega$.
\end{abstract}

\maketitle
	
\section{Introduction}
A very interesting area of nonlinear analysis is that of elliptic equations involving nonlocal fractional operators; see, among others, the quoted paper \cite{valpal} as well as the recent monograph \cite{MRS} and the references therein. In the case of Dirichlet boundary conditions, an impressive number of existence and multiplicity results has been obtained over the last few decades by using different technical approaches. On the contrary, fractional problems with mixed Dirichlet-Neumann boundary data have not been investigated much 
due to the additional analytic difficulties that naturally arise when dealing with the notion of normal derivative in this setting. Some recent interesting results can be found in \cite{Barrios2020, carcolleoort2020regularity, colort2019the, Leonori2018, Ortega2023}.

Moving along this second direction, in this paper we analyze the multiplicity of solutions to the following nonlocal problem
\begin{equation}\label{problem}\tag{$P_\lambda$}
\left\{
        \begin{tabular}{lcl}
        $(-\Delta)^su=\lambda u + f(x,u)$ & &in $\Omega$, \\[3pt]
        $\mkern+21muB(u)=0$& &on $\partial\Omega$,
        \end{tabular}
        \right.
\end{equation}
where $\Omega\subset\R^N$ is a bounded domain with a smooth boundary, $N>2s$ and $s\in(1/2,1)$. The fractionality range $\frac{1}{2} < s < 1$ is the correct one for mixed boundary problems due to the natural embedding of the associated functional space, see Remark \ref{rem:range_s}. Here, $(-\Delta)^s$ denotes the spectral fractional Laplace operator on $\Omega$ endowed with the mixed Dirichlet-Neumann boundary conditions
$$
B(u)=u\chi_{\SD} +\frac{\partial u}{\partial\nu}\chi_{\SN},
$$
where $\nu$ is the outward unit normal to $\partial\Omega$, $\chi_A$ denotes the characteristic function of the set $A\subset\partial\Omega$ and $\Omega$ satisfies the following set of assumptions:
\begin{itemize}
\item[$(\Omega_1)$] $\Omega\subset\R^N$ is a bounded Lipschitz domain;
\item[$(\Omega_2)$] $\SD$ and $\SN$ are smooth $(N-1)$ -- dimensional submanifolds of $\partial\Omega$;
\item[$(\Omega_3)$] $\SD$ is a closed manifold with positive $(N-1)$ -- dimensional Lebesgue measure, say $|\SD|=\alpha\in(0,|\partial\Omega|)$;
\item[$(\Omega_4)$] $\SD\cap\SN=\emptyset$, $\SD\cup\SN=\partial\Omega$ and $\SD\cap\overline{\Sigma}_\mathcal{N}=\Gamma$, where $\Gamma$ is a smooth $(N-2)$--dimensional submanifold of $\partial\Omega$.
\end{itemize}

The nonlinearity $f:\Omega\times\R\to\R$ is supposed to be a Carathéodory function satisfying
\begin{itemize}
	\item[$(f_1)$] $|f(x,t)|\leq a_1\left(1+|t|^{q-1}\right)$ for a.e. $x\in\Omega$, for all $t\in\R$, and for some $a_1\in \R$ and $q\in\left(2,\frac{2N}{N-2s}\right)$;
	 \item[$(f_2)$] $\displaystyle F(x,t):=\int_0^t f(x,\tau)d\tau \geq a_2|t|^q-a_3$, for a.e. $x\in\Omega$, for all $t\in\R$ and for some $a_2,a_3\in \R$;
	 \item[$(f_3)$] $\displaystyle\lim_{t\to 0}\frac{f(x,t)}{|t|}=0$, uniformly in $x\in\Omega$;
	 \item[$(f_4)$] $0<qF(x,t)\leq tf(x,t)$ for a.e. $x\in\Omega$ and for all $t\in\R\setminus\{0\}$.
\end{itemize}

As a prototype of nonlinearity fulfilling $(f_1)-(f_4)$, one can choose $f(x,t)=\beta(x)|t|^{q-2}$, where $\beta\in L^\infty(\Omega)$, $\inf_\Omega \beta>0$, and $q$ as above.\par
\smallskip
Our main result reads as follows:
\begin{theorem}\label{mainresult}
Assume $(\Omega_1)-(\Omega_4)$ and $(f_1)-(f_4)$. Then, for every eigenvalue $\Lambda_k$, $k\geq 2$, of the problem
\begin{equation}\label{eigenproblem}
	\left\{
	\begin{array}{rl}
		(-\Delta)^{s} u = \lambda u  & \text{in } \Omega\\
		B(u)=0\mkern+12mu & \text{on } \partial\Omega,
	\end{array}
	\right.
\end{equation}
there exists $\delta_k>0$ such that \eqref{problem} has at least three non-trivial weak solutions for all  $\lambda\in(\Lambda_k-\delta_k,\Lambda_k)$.
\end{theorem}

The proof of Theorem~\ref{mainresult} relies upon both variational and topological arguments, more precisely, upon a $\nabla$-theorem due to Marino and Saccon contained in \cite{marsac1997some} and recalled here in Theorem~\ref{nablathm}. Although such technical arguments are similar to the ones previously used by several authors (see, for instance, \cite{mu0,marsac1997some, mu1,mu4,mu2} and references therein), in our setting, due to the nonlocal nature of the problem, some additional difficulties naturally appear. For instance, in order to correctly encode the Neumann-Dirichlet boundary conditions in the variational formulation of the problem, a careful analysis of the involved Sobolev spaces is necessary (see Section \ref{functionalsettings}). To the best of our knowledge, no results concerning nonlocal fractional problems with mixed boundary conditions comparable to Theorem~\ref{mainresult} are available in the literature. 

This paper is organized as follows. In Section \ref{functionalsettings} we introduce the variational set-up of problem \eqref{problem}, together 
with some preliminary results. In Section \ref{nablasect} we discuss the compactness and geometric properties of the energy functional $I_\lambda$ associated with \eqref{problem}. In particular, we prove that $I_\lambda$ satisfies the $\nabla$-condition, one of the key points in Theorem~\ref{nablathm}. The last section is devoted to the proof of Theorem \ref{mainresult}.

\section{Functional framework}\label{functionalsettings}
We begin by first recalling the definition of the fractional Laplace operator, based on the spectral decomposition of the classical Laplace operator under mixed boundary conditions. Let $(\lambda_i,\varphi_i)$ be the eigenvalues and the eigenfunctions (normalized with respect to the $L^2(\Omega)$-norm) of $-\Delta$ with homogeneous mixed Dirichlet-Neumann boundary conditions, respectively. Then $(\lambda_i^s,\varphi_i)$ are the eigenvalues and the eigenfunctions of $(-\Delta)^s$. Consequently, given two smooth functions 
$$
u_i(x)=\sum_{j\geq1}\langle u_i,\varphi_j\rangle_{2}\varphi_j(x), \quad i=1,2,
$$
where $\langle u,v\rangle_2=\displaystyle\int_{\Omega}uv\,dx$ is the standard scalar product on $L^2(\Omega)$, one has
\begin{equation}\label{pre_prod}
\langle(-\Delta)^s u_1, u_2\rangle_{2} = \sum_{j\ge 1} \lambda_j^s\langle u_1,\varphi_j\rangle_{2} \langle u_2,\varphi_j\rangle_{2},
\end{equation}
that is, the action of the fractional operator on a smooth function $u_1$ is given by
\begin{equation}\label{action}
(-\Delta)^su_1=\sum_{j\ge 1} \lambda_j^s\langle u_1,\varphi_j\rangle_{2}\varphi_j.
\end{equation}
As a consequence, the {\it spectral} fractional Laplace operator $(-\Delta)^s$ is well defined in the following Hilbert space of functions that vanish on $\SD$:
\begin{equation*}
H_{\Sigma_{\mathcal{D}}}^s(\Omega)\vcentcolon=\left\{u=\sum_{j\ge 1} a_j\varphi_j\in L^2(\Omega):\ u=0\ \text{on }\Sigma_{\mathcal{D}},\ \|u\|_{H_{\Sigma_{\mathcal{D}}^s}(\Omega)}^2:=
\sum_{j\ge 1} a_j^2\lambda_j^s<+\infty\right\}.
\end{equation*}
Thus, given $u\in H_{\SD}^s(\Omega)$, it follows by definition that
\begin{equation}\label{H_norm}
\left\| u\right\|_{H_{\SD}^s(\Omega)}=\left\| (-\Delta)^{s/2} u\right\|_{L^2(\Omega)}.
\end{equation}
Moreover, taking in mind \eqref{pre_prod}, the norm $\left\| \cdot\right\|_{H_{\SD}^s(\Omega)}$ is induced by the scalar product
\begin{equation}\label{sc_prod}
\langle u_1,u_2\rangle_{H_{\Sigma_{\mathcal{D}}}^s}=\left\langle (-\Delta)^su_1,u_2\right\rangle_{2}=\langle (-\Delta)^{\frac{s}{2}}u_1,(-\Delta)^{\frac{s}{2}}u_2\rangle_{2}=\left\langle u_1,(-\Delta)^su_2\right\rangle_{2},
\end{equation}
for all $u_1,u_2\in H_{\SD}^s(\Omega)$. The above chain of equalities can be simply stated as an integration-by-parts like formula.

\begin{remark}\label{rem:range_s}
{\rm As it is proved in \cite[Theorem 11.1]{Lions1972}, if $0<s\le \frac{1}{2}$, then $H_0^s(\Omega)=H^s(\Omega)$ and, thus, also
$H_{\Sigma_{\mathcal{D}}}^s(\Omega)=H^s(\Omega)$, while for $\frac 12<s<1$, $H_0^s(\Omega)\subsetneq H^s(\Omega)$. Therefore, the range $\frac 12<s<1$ ensures $H_{\Sigma_{\mathcal{D}}}^s(\Omega)\subsetneq H^s(\Omega)$ and it provides us with the appropriate functional space for the mixed boundary problem \eqref{problem}.
}
\end{remark}

As commented before, this spectral definition of $(-\Delta)^s$ allows us to integrate by parts in the proper spaces, so that a natural definition of weak solution to \eqref{problem} is the following.

\begin{definition}
{\rm	
We say that $u\in H_{\SD}^s(\Omega)$ is a weak solution to \eqref{problem} if, for all $v\in H_{\SD}^s(\Omega)$,
\begin{equation}\label{weak_solution}
\int_\Omega (-\Delta)^{s/2}u (-\Delta)^{s/2} v\, dx =\lambda\int_\Omega uv \; dx + \int_\Omega f(x,u)v\, dx,
\end{equation}
that is,
\begin{equation*}
	\langle u,v \rangle_{H_{\Sigma_{\mathcal{D}}}^s}=\langle \lambda u +f(\cdot,u), v\rangle_{L^2(\Omega)}.
\end{equation*}
}
\end{definition}
We observe that \eqref{weak_solution} is well-defined because of the embedding $H_{\Sigma_{\mathcal{D}}}^s(\Omega)\hookrightarrow L^{2_s^*}(\Omega)$, so given $u\in H_{\Sigma_{\mathcal{D}}}^s(\Omega)$, by $(f_2)$ one has $\lambda u+f(x,u)\in L^{\frac{2N}{N+2s}}\hookrightarrow \left(H_{\Sigma_{\mathcal{D}}}^s(\Omega)\right)'$.\\

\noindent The energy functional $I_\lambda:H_{\Sigma_{\mathcal{D}}}^s(\Omega)\to\R$ associated with \eqref{problem} is given by
\begin{equation}\label{functional_down}
I_\lambda(u):=\frac12\int_{\Omega}|(-\Delta)^{\frac{s}{2}}u|^{2}dx-\frac{\lambda}{2}\int_{\Omega}u^2dx-\int_{\Omega}F(x,u)dx.
\end{equation}
Indeed, as one can promptly see,
\begin{equation*}
\langle I_\lambda'(u), v\rangle_{H^{-s}}  = \int_{\Omega}(-\Delta)^{\frac{s}{2}}u(-\Delta)^{\frac{s}{2}}v\,dx-\lambda\int_{\Omega}uv\,dx  - \int_{\Omega}f(x,u)v\,dx,
\end{equation*}
for all $u,v\in H_{\Sigma_{\mathcal{D}}}^s(\Omega)$, where $\langle\cdot,\cdot\rangle_{H^{-s}}$ denotes the duality between $H_{\Sigma_{\mathcal{D}}}^s(\Omega)$ and the dual space $H^{-s}(\Omega)\vcentcolon=\left(H_{\Sigma_{\mathcal{D}}}^s(\Omega)\right)'$.

In addition to the definition given above, there is another equivalent characterization of the fractional Laplacian in terms of the so-called Dirichlet-to-Neumann operator. Indeed, the (functional) square root of the Laplacian acting on a function $u$ in the whole space $\mathbb{R}^{N}$, can be computed as the normal derivative on the boundary of its harmonic extension to the upper half-space $\mathbb{R}_+^{N+1}$. Based on this idea, Caffarelli and Silvestre (cf. \cite{cafsil2007an}) proved that $(-\Delta)^s$ can be realized in a local way by using one more variable and the so-called {\it $s$-harmonic extension}. This was later extended to bounded domains in \cite{bracoldepsan2013a,Cabre2010,Capella2011}. We illustrate next such local characterization for the sake of completeness, although we point out that we will completely work on the nonlocal setting avoiding the use of the local realization provided by the $s$-harmonic extension.

Consider the cylinder $\mathcal C_\Omega:=\Omega\times(0,+\infty)\subset\R_+^{N+1}$ and denote by $\partial_L\mathcal{C}_\Omega:=\partial\Omega\times[0,+\infty)$ its lateral boundary. Set also $\SD^*:=\SD\times[0,+\infty)$,  $\SN^*:=\SN\times[0,+\infty)$ and $\Gamma^*:=\Gamma\times[0,+\infty)$. Note that, by construction, $\SD^*\cap\SN^*=\emptyset$, $\SD^*\cup\SN^*=\partial_L\mathcal C_\Omega$ and $\SD^*\cap\overline{\SN^*}=\Gamma^*$.

Given $u\in H_{\SD}^s(\Omega)$, we define its $s$-harmonic extension $w(x,y)=E_s[u(x)]$ as the solution to the problem
\begin{equation*}
        \left\{
        \begin{array}{rlcl}\label{extprobl}
           -\text{div}\left(y^{1-2s}\nabla w(x,y) \right)&\!\!\!\!=0  & & \mbox{ in } \mathcal{C}_{\Omega} , \\
          B(w(x,y))&\!\!\!\!=0   & & \mbox{ on } \partial_L\mathcal{C}_{\Omega} , \\
          w(x,0)&\!\!\!\!=u(x)  & &  \mbox{ on } \Omega\times\{y=0\} ,
        \end{array}
        \right.
\end{equation*}
where $\displaystyle B(w)=w\chi_{\SD^*} +\frac{\partial w}{\partial\nu}\chi_{\SN^*}$, being $\nu$, by a slight abuse of notation, the outward unit normal to $\partial_L\mathcal C_\Omega$. The extension function $w$ belongs to the space
$$
 \mathcal{X}_{\Sigma_{\mathcal{D}}^*}^s(\mathcal{C}_{\Omega})\vcentcolon=\overline{C_0^\infty(\Omega\cup\SN)\times[0,+\infty)}^{\left\| \cdot\right\|_{\mathcal{X}_{\Sigma_{\mathcal{D}}^*}^s(\mathcal{C}_{\Omega})}},
$$
where we define
\begin{equation}\label{norma}
\left\| \cdot \right\|_{\mathcal{X}_{\Sigma_{\mathcal{D}}^*}^s(\mathcal{C}_{\Omega})}\vcentcolon=k_s\int_{\mathcal C_\Omega} y^{1-2s}|\nabla (\cdot)|^2dxdy,
\end{equation}
$k_s:=2^{2s-1}\frac{\Gamma(s)}{\Gamma(1-s)}$. The space $\mathcal{X}_{\Sigma_{\mathcal{D}}^*}^s(\mathcal{C}_{\Omega})$ is a Hilbert space equipped with the
norm $\left\| \cdot \right\|_{\mathcal{X}_{\Sigma_{\mathcal{D}}^*}^s(\mathcal{C}_{\Omega})}$, which is induced by the scalar product
$$
\left\langle w,z\right\rangle_{\mathcal{X}_{\Sigma_{\mathcal{D}}^*}^s(\mathcal{C}_{\Omega})}\vcentcolon=k_s \int_{\mathcal C_\Omega}y^{1-2s}\left\langle \nabla w,\nabla z\right\rangle dx dy,
$$
$w,z\in \mathcal{X}_{\Sigma_{\mathcal{D}}^*}^s(\mathcal{C}_{\Omega})$.
Moreover the following inclusions hold
\begin{equation} \label{inclusions}
\mathcal{X}_0^s(\mathcal{C}_{\Omega}) \subsetneq \mathcal{X}_{\Sigma_{\mathcal{D}}^*}^s(\mathcal{C}_{\Omega}) \subsetneq \mathcal{X}^s(\mathcal{C}_{\Omega}),
\end{equation}
being  $\mathcal{X}_0^s(\mathcal{C}_{\Omega})$ the space of functions that belong to $\mathcal{X}^s(\mathcal{C}_{\Omega})\equiv H^1(\mathcal{C}_{\Omega},y^{1-2s}dxdy)$ and vanish on $\partial_L\mathcal{C}_{\Omega}$. The key point of the extension technique is the following localization formula for the fractional Laplacian  (cf. \cite{bracoldepsan2013a,cafsil2007an}),
$$
\frac{\partial w}{\partial \nu^s}=-k_s\lim_{y\to 0^+}y^{1-2s}\frac{\partial w}{\partial y}=(-\Delta)^s u(x),
$$
thanks to which \eqref{problem} can be restated as
\begin{equation}\label{extorigprobl}        \tag{$P_{\lambda}^*$}
 \left\{
        \begin{array}{rlcl}
           -\text{div}(y^{1-2s}\nabla w )&\!\!\!\!=0  & & \mbox{ in } \mathcal{C}_{\Omega} , \\
          B(w)&\!\!\!\!=0   & & \mbox{ on } \partial_L\mathcal{C}_{\Omega} , \\
         \displaystyle\frac{\partial w}{\partial \nu^s}&\!\!\!\!=\lambda w(x,0)+f(x,w(x,0)) & &  \mbox{ on } \Omega\times\{y=0\}.
        \end{array}
        \right.
\end{equation}

\begin{definition}
{\rm
We say that $w\in \mathcal{X}_{\Sigma_{\mathcal{D}}^*}^s(\mathcal{C}_{\Omega})$ is a weak solution to \eqref{extorigprobl} if, for all $z\in \mathcal{X}_{\Sigma_{\mathcal{D}}^*}^s(\mathcal{C}_{\Omega})$,
\begin{equation*}
	k_s\int_{\mathcal C_\Omega} y^{1-2s} \nabla w\cdot \nabla z dxdy =\lambda\int_\Omega w(x,0)z(x,0)dx + \int_\Omega f(x,w(x,0))z(x,0)dx.
\end{equation*}
}
\end{definition}

In the light of the above discussion, if  $w\in \mathcal{X}_{\Sigma_{\mathcal{D}}^*}^s(\mathcal{C}_{\Omega})$ weakly solves \eqref{extorigprobl}, then its trace
$u(x)=\text{Tr}[w](x):=w(x,0)$ belongs to
$H_{\Sigma_{\mathcal{D}}}^s(\Omega)$ and is a weak solution to
problem \eqref{problem}. Vice versa, if $u\in
H_{\Sigma_{\mathcal{D}}}^s(\Omega)$ is a solution to \eqref{problem}, then $w=E_s[u]\in
\mathcal{X}_{\Sigma_{\mathcal{D}}^*}^s(\mathcal{C}_{\Omega})$ satisfies \eqref{extorigprobl}. Thus, both formulations are equivalent and the {\it extension operator}
$$
E_s: H_{\Sigma_{\mathcal{D}}}^s(\Omega) \to \mathcal{X}_{\Sigma_{\mathcal{D}}^*}^s(\mathcal{C}_{\Omega}),
$$
allows us to switch between each other. Moreover, due to the choice of the constant $k_s$, the extension operator $E_s$ is an isometry (cf. \cite{bracoldepsan2013a, cafsil2007an}),
\begin{equation}\label{norma2}
\|E_s[u] \|_{\mathcal{X}_{\Sigma_{\mathcal{D}}^*}^s(\mathcal{C}_{\Omega})}=
\|u\|_{H_{\Sigma_{\mathcal{D}}}^s(\Omega)}, \quad \text{for all } u\in H_{\Sigma_{\mathcal{D}}}^s(\Omega).
\end{equation}
Finally, the energy functional corresponding to $(P_{\lambda}^*)$ is defined by
\begin{equation*}
J_\lambda(w)=\frac{\kappa_s}{2}\int_{\mathcal{C}_{\Omega}}y^{1-2s}|\nabla w|^2dxdy-\frac{\lambda}{2}\int_{\Omega}|w(x,0)|^2dx-\int_{\Omega}F(x,w(x,0))dx,
\end{equation*}
and, plainly, critical points of $J_\lambda$ in
$\mathcal{X}_{\Sigma_{\mathcal{D}}^*}^s(\mathcal{C}_{\Omega})$
correspond to critical points of $I_\lambda$ in
$H_{\Sigma_{\mathcal{D}}}^s(\Omega)$. Moreover, minima of $J_\lambda$ are also minima of $I_\lambda$ (the proof of this fact is similar to the one of the pure Dirichlet case, see \cite[Proposition 3.1]{Barrios2012}).

When one considers Dirichlet boundary conditions the following \textit{trace inequality} holds (cf. \cite[Theorem 4.4]{bracoldepsan2013a}): there exists $C=C(N,s,r,|\Omega|)>0$ such that, for all $w\in\mathcal{X}_0^s(\mathcal{C}_{\Omega})$,
\begin{equation}\label{sobext}
 C\left(\int_{\Omega}|w(x,0)|^r dx\right)^{\frac{2}{r}}\leq k_s\int_{\mathcal{C}_{\Omega}}y^{1-2s}|\nabla w(x,y)|^2dxdy,
\end{equation}
for $r\in[1,2_s^*]$, $2_s^*:=\frac{2N}{N-2s}$, and $N>2s$. Because of \eqref{norma2}, inequality \eqref{sobext} is equivalent to the fractional Sobolev inequality: for all $u\in H_{0}^s(\Omega)$ and $r\in[1,2^*_s]$, $N>2s$, one has
\begin{equation}\label{sobolev}
  C\left(\int_{\Omega}|u|^rdx\right)^{\frac{2}{r}}\leq \int_{\Omega}\left|(-\Delta)^{\frac{s}2}u\right|^2dx.
\end{equation}
If $r=2_s^*$ the best constant in \eqref{sobolev} (and, by \eqref{norma2}, also in \eqref{sobext}), 
denoted by $S(N,s)$, is independent of $\Omega$ and its exact value is given by
\begin{equation*}
S(N,s)=2^{2s}\pi^s\frac{\Gamma\left(\frac{N+2s}{2}\right)}{\Gamma\left(\frac{N-2s}{2}\right)}\left(\frac{\Gamma(\frac{N}{2})}{\Gamma(N)}\right)^{\frac{2s}{N}}.
\end{equation*}
Since it is not achieved in any bounded domain, one has
\begin{equation*}
S(N,s)\left(\int_{\mathbb{R}^N}|w(x,0)|^{\frac{2N}{N-2s}}dx\right)^{\frac{N-2s}{N}}\leq k_s\int_{\mathbb{R}_{+}^{N+1}}y^{1-2s}|\nabla w(x,y)|^2dxdy,
\end{equation*}
for all $w\in \mathcal{X}^s(\mathbb{R}_{+}^{N+1})$ where $\mathcal{X}^s(\mathbb{R}_{+}^{N+1})\vcentcolon=\overline{\mathcal{C}^{\infty}(\mathbb{R}^N\times[0,\infty))}^{\|\cdot\|_{\mathcal{X}^s(\mathbb{R}_{+}^{N+1})}}$, with $\|\cdot\|_{\mathcal{X}^s(\mathbb{R}_{+}^{N+1})}$ defined as
\eqref{norma} replacing $\mathcal{C}_\Omega$ by $\mathbb{R}_{+}^{N+1}$. In the whole space, the latter inequality turns into equality for the family $w_{\varepsilon}= E_s[u_{\varepsilon}]$,
\begin{equation}\label{eq:sob_extremal}
u_{\varepsilon}(x)=\frac{\varepsilon^{\frac{N-2s}{2}}}{(\varepsilon^2+|x|^2)^{\frac{N-2s}{2}}},
\end{equation}
with arbitrary $\varepsilon>0$, (cf. \cite{bracoldepsan2013a}).

When mixed boundary conditions are considered the situation is quite similar since the Dirichlet condition is imposed on a set $\Sigma_{\mathcal{D}} \subset \partial \Omega$ such that $0<\alpha<|\partial\Omega|$.

\begin{definition}\label{defi_sob_const}
{\rm
The Sobolev constant relative to the Dirichlet boundary $\Sigma_{\mathcal{D}}$ is defined by
\begin{equation*}
\widetilde{S}(\Sigma_{\mathcal{D}})=\inf_{\substack{u\in
H_{\Sigma_{\mathcal{D}}}^s(\Omega)\\ u\not\equiv
0}}\frac{\|u\|_{H_{\Sigma_{\mathcal{D}}}^s(\Omega)}^2}{\|u\|_{L^{2_s^*}(\Omega)}^2}=\inf_{\substack{w\in \mathcal{X}_{\Sigma_{\mathcal{D}}^*}^s(\mathcal{C}_{\Omega})\\
w\not\equiv
0}}\frac{\|w\|_{\mathcal{X}_{\Sigma_{\mathcal{D}}}^s(\mathcal{C}_{\Omega})}^2}{\|w(\cdot,0)\|_{L^{2_s^*}(\Omega)}^2}.
\end{equation*}
}
\end{definition}

Since $0<\alpha<|\partial\Omega|$, by the inclusions \eqref{inclusions}, one has
\begin{equation}\label{const}
0<\widetilde{S}(\Sigma_{\mathcal{D}})<S(N,s).
\end{equation}
Actually, by \cite[Proposition 3.6]{colort2019the} it turns out that $\widetilde{S}(\Sigma_{\mathcal{D}})\leq 2^{-\frac{2s}{N}}S(N,s)$ and, by \cite[Theorem 2.9]{colort2019the}, if $\widetilde{S}(\Sigma_{\mathcal{D}})<2^{-\frac{2s}{N}}S(N,s)$, then $\widetilde{S}(\Sigma_{\mathcal{D}})$ is attained.

\begin{remark}\label{remark_att}
{\rm Regarding the attainability of the constant $\widetilde{S}(\Sigma_{\mathcal{D}})$ it is worth to notice the following. Due to the spectral definition of $(-\Delta)^s$ we have $\widetilde{S}(\Sigma_{\mathcal{D}})\leq|\Omega|^{\frac{2s}{N}}\lambda_1^s(\alpha)$, where $\lambda_1(\alpha)$ is the first eigenvalue of $(-\Delta)$ endowed with mixed boundary conditions on the sets $\Sigma_{\mathcal{D}}=\Sigma_{\mathcal{D}}(\alpha)$ and $\Sigma_{\mathcal{N}}= \Sigma_{\mathcal{N}}(\alpha)$. Since $\lambda_1(\alpha)\to0$ as $\alpha\to0^+$, (cf. \cite[Lemma 4.3]{colper2003semilinear}), we have $\widetilde{S}(\Sigma_{\mathcal{D}})\to0$ as $\alpha\to0^+$. Thus, taking the Dirichlet part small enough, the constant $\widetilde{S}(\Sigma_{\mathcal{D}})$ is attained independently of the geometry of $\Omega$. This is in contrast to the Dirichlet case as, by a Pohozaev-type identity, it is well known that the Sobolev constant $S(N,s)$ is not attained if the domain $\Omega$ is starshaped.
}
\end{remark}

Because of \eqref{norma2} and \eqref{const}, it follows that, for all $w\in \mathcal{X}_{\Sigma_{\mathcal{D}}^*}^s(\mathcal{C}_{\Omega})$,
\begin{equation*}
\widetilde{S}(\Sigma_{\mathcal{D}})\left(\int_\Omega |w(x,0)|^{2^*_s} dx\right)^{\frac{2}{2^*_s}}\leq\|w(x,0)\|_{H_{\Sigma_{\mathcal{D}}}^s(\Omega)}^2=\|E_s[w(x,0)]\|_{\mathcal{X}_{\Sigma_{\mathcal{D}}}^s(\mathcal{C}_{\Omega})}^2.
\end{equation*}
This Sobolev-type inequality provides us with a trace inequality in the mixed boundary data framework.

\begin{lemma}\cite[Lemma 2.4]{colort2019the}\label{lem:traceineq}
For all $w\in \mathcal{X}_{\Sigma_{\mathcal{D}}^*}^s(\mathcal{C}_{\Omega})$, one has
\begin{equation*}
\widetilde{S}(\Sigma_{\mathcal{D}})\left(\int_\Omega|w(x,0)|^{2^*_s}  dx\right)^{\frac{2}{2^*_s}}\leq k_s\int_{\mathcal{C}_{\Omega}} y^{1-2s} |\nabla w(x,y)|^2 dxdy.
\end{equation*}
In an equivalent way, for all $u\in H_{\SD}^s(\Omega)$ it holds
\begin{equation}\label{mix_sobolev}
	\widetilde{S}(\Sigma_{\mathcal{D}})\left(\int_{\Omega}|u|^{2^*_s}dx\right)^{\frac{2}{2^*_s}}\leq \int_{\Omega}|(-\Delta)^{\frac{s}2}u|^2dx.
\end{equation}
\end{lemma}
In what follows, we shall use the same symbols $c,C$, or some subscripted version of them, to denote possibly different positive constants. Moreover, if $A\subseteq H_{\SD}^s(\Omega)$ and $r>0$, we shall indicate by $B_r(A)$ and $\partial B_r(A)$ the sets
\begin{equation*}
B_r(A):=\{u\in A: \left\|u\right\|_{H_{\SD}^s(\Omega)}\leq r\}\qquad\text{and}\qquad \partial B_r(A):=\{u\in A: \left\|u\right\|_{H_{\SD}^s(\Omega)}= r\}
\end{equation*}
respectively.
\section{Linking geometry, $\nabla$- and $(PS)$-conditions}\label{nablasect}
The main tool we rely upon for the proof of Theorem \ref{mainresult} is  a "mixed-type" theorem (known also as a $\nabla$-theorem) due to Marino and Saccon, in which properties of both a functional and its gradient are used (\cite[Theorem 2.10]{marsac1997some}). Basically, a linking structure together with a suitable condition of the gradient on certain subsets permit to deduce the existence of multiple critical points. Before stating this theorem, let us fix some notation and give some definition.

If $X$ be a Hilbert space and $M$ a closed subspace of $X$, we denote by $\Pi_M:X\to M$ the orthogonal projection of $X$ on $M$ and by $d(u,M)\vcentcolon=\inf_{v\in M}d(u,v)$ the distance of $u\in X$ from $M$. In addition, let $I:X\to\R$ be a $C^1$ functional and $a,b\in\R\cup\{-\infty,+\infty\}$. With these ingredients at hand, we give the following

\begin{definition}
{\rm	
We say that $(I,M,a,b)$ satisfy the $\nabla$-condition, abbreviated by $(\nabla)(I,M,a,b)$, if there exists $\gamma>0$ such that
$$
\inf\Big\{\| \Pi_M\nabla I(u)\|_{X}: u\in M, \; a\leq I(u)\leq b, \;d(u,M)\leq\gamma \Big\} >0.
$$
}
\end{definition}

\noindent We also recall that $I$ is said to satisfy the Palais-Smale condition at level $c\in\R$, $(PS)_c$ for short, if any sequence $\{u_j\}\subset X$ satisfying
$$
I(u_j)\to c, \quad I'(u_j)\to 0 \text{ in } X^*, \quad \text{as } j\to\infty,
$$
has a convergent subsequence.

\begin{theorem}\label{nablathm}
Let $X$ be a Hilbert space and let $X_i$, $i=1,2,3$, be three subspaces of $X$ such that $X=X_1\oplus X_2 \oplus X_3$, with $dim \ X_i <\infty$ for $i=1,2$. Denote by $\Pi_i:X\to X_i$ the orthogonal projection of $X$ on $X_i$. Let $I:X\to\R$ be of class $C^1$, $\varrho,\varrho',\varrho'',\varrho_1$ be such that $\varrho_1>0$, $0\leq\varrho'<\varrho<\varrho''$, and define
\begin{align*}
\Delta&\vcentcolon=\left\lbrace u\in X_1\oplus X_2: \varrho'\leq \left\|\Pi_2 u \right\|\leq\varrho'', \left\| \Pi_1 u\right\|\leq\varrho_1\right\rbrace,\\
T&\vcentcolon=\partial_{X_1\oplus X_2}\Delta.
\end{align*}
Suppose that
\begin{equation}
a'\vcentcolon=\sup I(T) <\inf I(\partial B_\varrho(X_2\oplus X_3))=\vcentcolon a''.
\end{equation}
Let $a,b\in\R$ be such that $a'<a<a''$ and $b>\sup I(\Delta)$, and assume that $(\nabla)(I,X_1\oplus X_3,a,b)$ and $(PS)_c$ hold, the latter for all $c\in[a,b]$.

Then, $I$ has at least two critical points in $I^{-1}([a,b])$. Moreover, if in addition we have $\inf I(B_\varrho(X_2\oplus X_3))>-\infty$ 
and $(PS)_c$ holds for all $c\in[a_1,b]$, with $a_1<\inf I(B_\varrho(X_2\oplus X_3))$ 
then $I$ has another critical level in $[a_1,a']$.
\end{theorem}

In order to use Theorem \ref{nablathm} to prove Theorem \ref{mainresult} we start by checking that the energy functional $I_\lambda$ defined in \eqref{functional_down} has the linking geometry stated in Theorem \ref{nablathm}.

Let $\{\Lambda_k\}$ be the eigenvalues of $(-\Delta)^s$ in $\Omega$ with mixed Dirichlet-Neumann conditions, in increasing order and counted with their multiplicity,
$$
0<\Lambda_1<\Lambda_2\leq\ldots\Lambda_k\leq\Lambda_{k+1}\leq\ldots, \quad \Lambda_k\to +\infty \text{ as } +\to\infty,
$$
where, by definition, $\Lambda_k=\lambda_k^s$, $k\in\N$. For every $k\geq 1$, set
$$
\mathbb{H}_k\vcentcolon=\text{span}\{\varphi_1,\ldots,\varphi_k\}
$$
and
$$
\mathbb{P}_k\vcentcolon=\{u\in H_{\SD}^s(\Omega): \left\langle u,\varphi_j\right\rangle_{H_{\SD}^s}=0, \text{ for all } j=1,\ldots,k\},
$$
where $\varphi_k$ is the eigenfunction corresponding to $\Lambda_k$. In this setting, we have the following minimax characterization of the eigenvalues.
\begin{lemma}\label{varcharacteigen}
For any $k\in\mathbb{N}$, one has
$$
\Lambda_k=\inf_{u\in \mathbb{P}_{k-1}}\frac{\left\|u\right\|_{H_{\SD}^s(\Omega)}^2}{\left\| u\right\|_{L^2(\Omega)}^2}=\sup_{u\in \mathbb{H}_{k}}\frac{\left\|u\right\|_{H_{\SD}^s(\Omega)}^2}{\left\| u\right\|_{L^2(\Omega)}^2}.
$$
\end{lemma}

\begin{proof}
By \eqref{action} and \eqref{H_norm} we deduce
$$
\left\| u\right\|_{H_{\SD}^s(\Omega)}^2=\sum_{j=1}^\infty \lambda_j^s \left\langle u,\varphi_j\right\rangle _{2}^2.
$$
Then,
\begin{align*}
\left\| u\right\|_{H_{\SD}^s(\Omega)}^2  &= \sum_{j=k}^\infty \lambda_j^s \left\langle u,\varphi_j\right\rangle _{2}^2 \geq \lambda_k^s\sum_{j=k}^\infty \left\langle u,\varphi_j\right\rangle _{2}^2 =\lambda_k^s \left\|u\right\|_{L^2(\Omega)}^2, \\
\left\| v\right\|_{H_{\SD}^s(\Omega)}^2 &= \sum_{j=1}^k \lambda_j^s \left\langle v,\varphi_j\right\rangle _{2}^2 \leq \lambda_k^s\sum_{j=1}^k \left\langle v,\varphi_j\right\rangle _{2}^2 =\lambda_k^s \left\|v\right\|_{L^2(\Omega)}^2,
\end{align*}
for all $u\in \mathbb{P}_{k-1}$ and for all $v\in \mathbb{H}_k$. In both the above relations the equality occurs iff $u,v\in\text{span}\{\varphi_k\}$.
\end{proof}	

The next result will be useful for what follows.

\begin{proposition}\label{phikbounded}
		Let $\lambda>0$ and let $\varphi\in H_{\SD}^1(\Omega)$ be a solution to
		\begin{equation}\label{eigenproblemlaplacian}
			\left\{
			\begin{array}{ll}
				-\Delta u = \lambda u  & \text{in } \Omega\\
				B(u)=0\mkern+12mu & \text{on } \partial\Omega.
			\end{array}
			\right.
		\end{equation}
		Then  $\varphi\in L^\infty(\Omega)$.
	\end{proposition}
\begin{proof}
		The proof follows by a Moser-type iteration process combined with the Stroock-Varopoulos inequality \cite{Stroock1985}.
		Let $(\lambda,\varphi)\in(0,+\infty)\times H_{\SD}^1(\Omega)$ be an eigenpair of \eqref{eigenproblemlaplacian}. By Sobolev inequality for mixed boundary problems
		\begin{equation*}
			S(\Sigma_{\mathcal{D}})\|u\|_{\frac{2N}{N-2}} \le \|\nabla u\|_2, \quad \text{for all } u\in H_{\Sigma_{\mathcal{D}}}^1(\Omega),
		\end{equation*}
		where $S(\Sigma_{\mathcal{D}})$ is the Sobolev constant relative to $\Sigma_D$ (cf. \cite{liopactri1988best}), for each $p>1$, one has
		\begin{equation*}
			\begin{split}
				\lambda\int_\Omega |\varphi|^p dx &= \int_\Omega |\varphi|^{p-2}\varphi(-\Delta \varphi)dx
				=\int_\Omega \nabla (|\varphi|^{p-2}\varphi)\cdot \nabla\varphi dx\\
				&=(p-1)\int_\Omega |\varphi|^{p-2}(\nabla\varphi\cdot \nabla\varphi) dx\\
				  &=\frac{4(p-1)}{p^2}\int_\Omega \left(\frac{p}{2}|\varphi|^{\frac{p-4}{2}}\nabla\varphi\right)\cdot \left(\frac{p}{2}|\varphi|^{\frac{p-4}{2}}\nabla\varphi \right)dx\\
				 &= \frac{4(p-1)}{p^2}\int_\Omega |\nabla |\varphi|^{p/2}|^2 dx\\
				&\ge \frac{4(p-1)}{p^2}S(\Sigma_{\mathcal{D}})\|\varphi\|_{\frac{Np}{N-2}}^p.
			\end{split}
		\end{equation*}		
Therefore, given $p_0>1$, if we set  $p_k:=p_0\left(\frac{N}{N-2}\right)^k$, $k\in\N$, it follows that
		\begin{equation*}
			\|\varphi\|_{p_{k+1}} \le \left(\frac{\lambda}{S(\Sigma_{\mathcal{D}})}\cdot\frac{p_k^2}{4(p_k-1)}\right)^{1/p_k}\|\varphi\|_{p_k}.
		\end{equation*}
		By iterating this scheme, we conclude that
		\begin{equation*}
			\|\varphi\|_\infty \le C(N,p_0)\lambda^\frac{N}{2p_0}\|\varphi\|_{p_0},
		\end{equation*}
		for a suitable constant $C(N,p_0)>0$ and the conclusion is achieved.

\end{proof}

We say that $\Lambda_k$, $k\geq 2$, has multiplicity $m\in\N$ if
$$
\Lambda_{k-1}<\Lambda_k=\Lambda_{k+1}=\ldots=\Lambda_{k+m-1}<\Lambda_{k+m};
$$
in this case the eigenspace associated with $\Lambda_k$ coincides with $\text{span}\{\varphi_k,\ldots,\varphi_{k+m-1}\}$. We set
\begin{equation*}
X_1\vcentcolon=\mathbb{H}_{k-1}, \quad X_2\vcentcolon=\text{span}\{\varphi_k,\ldots,\varphi_{k+m-1}\},\quad X_3\vcentcolon=\mathbb{P}_{k+m-1}.
\end{equation*}

\begin{proposition}\label{prop_supinf}
Let $k\in\N$, $k\geq 2$, and let $m\in\N$ be such that $\Lambda_{k-1}<\lambda<\Lambda_k=\ldots=\Lambda_{k+m-1}<\Lambda_{k+m}$. Then there exist $R,\varrho\in\R$, with $R>\varrho>0$, such that
\begin{equation*}
l_1\vcentcolon=\sup_{u\in B_R(X_1)\cup\partial B_R(X_1\oplus X_2) } I_\lambda(u) < \inf_{u\in \partial B_{\varrho}(X_2\oplus X_3)} I_\lambda(u)\vcentcolon=l_2.
	\end{equation*}
\end{proposition}

\begin{proof}
By $(f_1)$ and $(f_3)$, for every $\varepsilon>0$ there exists $C_\varepsilon>0$ so that
\begin{equation}\label{pre}
|F(x,t)|\leq \varepsilon t^2 +C_\varepsilon|t|^q,
\end{equation}
for a.e. $x\in\Omega$ and every $t\in\R$. Thus, if $u\in X_2\oplus X_3=\mathbb{P}_{k-1}$, because of \eqref{pre}, Lemma \ref{varcharacteigen}, H\"older's inequality and the fractional Sobolev inequality \eqref{mix_sobolev} we get
\begin{align*}
I_\lambda(u) &\geq \frac{1}{2}\left\| u\right\|_{H_{\SD}^s(\Omega)}^2 -\frac{\lambda}{2}\left\|u\right\|_{L^2(\Omega)}^2-\varepsilon \left\|u\right\|_{L^2(\Omega)}^2 - C_\varepsilon\left\|u \right\|_{L^2(\Omega)}^q\\
&\geq \frac{1}{2}\left(1-\frac{\lambda}{\Lambda_k}\right)\left\| u\right\|_{H_{\SD}^s(\Omega)}^2-\varepsilon |\Omega|^{\frac{2s}{N}}\|u\|_{L^{2_s^*}(\Omega)}^{2} - C_\varepsilon|\Omega|^{\frac{2_s^*-q}{2_s^*}}\left\|u \right\|_{L^{2_s^*}(\Omega)}^q\\
&\geq \left( \frac{1}{2}\left( 1-\frac{\lambda}{\Lambda_k}\right)-\varepsilon\frac{|\Omega|^{\frac{2s}{N}}}{\widetilde{S}(\Sigma_{\mathcal{D}})}\right) \left\|u \right\|_{H_{\SD}^s(\Omega)}^2-C_\varepsilon\frac{|\Omega|^{\frac{2_s^*-q}{2_s^*}}}{[\widetilde{S}(\Sigma_{\mathcal{D}})]^{\frac{q}{2}}}\left\|u \right\|_{H_{\SD}^s(\Omega)}^q.
\end{align*}
Taking $\varepsilon$ small enough, it follows that
\begin{equation*}
I_\lambda(u)\geq c_1\| u\|_{H_{\SD}^s(\Omega)}^2\left(1-c_2\left\| u\right\|_{H_{\SD}^s(\Omega)}^{q-2}\right),
\end{equation*}
for suitable constants $c_1,c_2>0$. Then, if $u\in \partial B_{\varrho(X_2\oplus X_3)}$ with $\varrho\in\left(0,c_2^{-1/(q-2)}\right)$, we get
\begin{equation*}
I_\lambda(u)\geq c_1\varrho^2(1-c_2\varrho^{q-2})>0,
\end{equation*}
and, as a result, $l_2>0$. Now, let $u\in X_1=\mathbb{H}_{k-1}$, i.e.,
\begin{equation*}
u(x)=\sum_{i=1}^{k-1}\alpha_i\varphi_i(x),
\end{equation*}
with $\alpha_i=\langle u,\varphi_i\rangle_{2},\ i=1,2,\ldots,k-1$. By $(f_4)$ and Lemma \ref{varcharacteigen} one has
\begin{equation}\label{diseqinHk-1}
\begin{split}	
I_\lambda(u) &\leq\frac{\Lambda_{k-1}}{2}\left\| u\right\|_{L^2(\Omega)}^2-\frac{\lambda}{2} \left\| u \right\|_{L^2(\Omega)}^2 -\int_\Omega F(x,u(x))dx\\
&\leq \frac{\Lambda_{k-1}-\lambda}{2}\left\| u\right\|_{L^2(\Omega)}^2\leq 0,
\end{split}
\end{equation}
since $\Lambda_{k-1}<\lambda$. On the other hand, if $u\in X_1\oplus X_2=\mathbb{H}_k$, because of $(f_2)$, we deduce that
\begin{equation*}
I_\lambda(u)\leq \frac{1}{2}\left\| u\right\|_{H_{\SD}^s(\Omega)}^2 - a_3\left\| u \right\|_{L^q(\Omega)}^q +a_4|\Omega|.
\end{equation*}
Since $u\in X_1\oplus X_2=\mathbb{H}_k$ and $\mathbb{H}_k$ is finite-dimensional, for suitable $c_3,c_4>0$, one obtains
\begin{equation*}
c_3 \left\| u\right\|_{H^s_{\SD}(\Omega)}\leq\left\| u\right\|_{L^q(\Omega)} \leq c_4 \left\| u\right\|_{H^s_{\SD}(\Omega)}.
\end{equation*}
We then get
\begin{equation}\label{ineqilambda}
I_\lambda(u)\leq\frac{1}{2}\left\| u\right\|_{H_{\SD}^s(\Omega)}^2 - a_3 c_3 \left\| u\right\|_{H^s_{\SD}(\Omega)}^q +a_4|\Omega|.
\end{equation}
Taking $R>0$ large enough, we conclude that $I_\lambda(u)\leq 0$ for $u\in \partial B_R(X_1\oplus X_2)$. This, together with \eqref{diseqinHk-1} implies that $l_1\leq 0$.
\end{proof}

Now we address the fulfillment of the $\nabla$-condition for $I_\lambda$. This will require some preliminary steps.

\begin{lemma}\label{uniquecritpoint}
Let $k\in\N$, $k\geq 2$, and let $\Lambda_k$ be an eigenvalue of multiplicity $m\in\N$. Then for any $\sigma>0$ there exists $\varepsilon_\sigma>0$ such that, for any $\lambda\in[\Lambda_{k-1}+\sigma,\Lambda_{k+m}-\sigma]$, the unique critical point $u$ of $I_\lambda |_{\mathbb{H}_{k-1}}\oplus \mathbb{P}_{k+m-1}$ with $I_\lambda(u)\in[-\varepsilon_\sigma,\varepsilon_\sigma]$ is the trivial one.
\end{lemma}
\begin{proof}
For the sake of contradiction, let us assume that there exist $\tilde\sigma>0$, two sequences $\{\mu_j\}\subset[\Lambda_{k-1}+\tilde\sigma,\Lambda_{k+m}-\tilde\sigma]$ and $\{u_j\}\subset \mathbb{H}_{k-1}\oplus \mathbb{P}_{k+m-1}\setminus\{0\}$ such that
\begin{equation}\label{relazJmuj}
\begin{split}
	\left\langle I_{\mu_j}'(u_j),\psi\right\rangle_{H^{-s}} &=0 \quad \text{for any } \psi\in \mathbb{H}_{k-1}\oplus \mathbb{P}_{k+m-1},\ \text{and any } j\in\mathbb{N},\\
	I_{\mu_j}(u_j)& \to 0 \quad \text{as } j\to \infty.
\end{split}
\end{equation}
Testing the first equation of \eqref{relazJmuj} with $\psi=u_j$ and using $(f_4)$ it follows that, for any $j\in\mathbb{N}$,
\begin{align*}
0 & = \left\| u_j\right\|_{H_{\SD}^s(\Omega)}^2 -\mu_j \left\| u_j\right\|_{L^2(\Omega)}^2-\int_\Omega f(x,u_j(x))u_j(x)dx\\
& =2I_{\mu_j}(u_j) +2\int_\Omega F(x,u_j(x))dx-\int_{\Omega}f(x,u_j(x))u_j(x)dx\\
&\leq  2I_{\mu_j}(u_j) + (2-q)\int_\Omega F(x,u_j(x))dx\\
&\leq 2I_{\mu_j}(u_j),
\end{align*}
where the last inequality follows since $q>2$. Thus, we conclude
\begin{equation}\label{intFxUj}
0<\int_\Omega F(x,u_j(x))dx\leq \frac{2}{q-2}I_{\mu_j}(u_j) \to 0 \quad \text{as } j\to\infty.
\end{equation}
Now, let us write $u_j=u_j^{\mathbb{H}}+u_j^{\mathbb{P}}$, with $u_j^{\mathbb{H}}\in \mathbb{H}_{k-1}$ and $u_j^{\mathbb{P}}\in \mathbb{P}_{k+m-1}$. Testing the first equation of \eqref{relazJmuj} with $\psi=u_j^{\mathbb{H}}-u_j^{\mathbb{P}}$ and using Lemma \ref{varcharacteigen} we obtain
\begin{equation}\label{deltalambdakm}
	\begin{split}
	\int_\Omega f(x,u_j(x))(u_j^{\mathbb{H}}(x)-u_j^{\mathbb{P}}(x))dx & =\left\langle u_j,u_j^{\mathbb{H}} \right\rangle_{H_{\SD}^s} - \left\langle u_j,u_j^{\mathbb{P}} \right\rangle_{H_{\SD}^s}\\
	 &\ \ -\mu_j\int_\Omega u_j(x)\left(u_j^{\mathbb{H}}(x)-u_j^{\mathbb{P}}(x)\right)dx\\
	& =\left\| u_j^{\mathbb{H}}\right\|_{H_{\SD}^s}^2-\left\| u_j^{\mathbb{P}}\right\|_{H_{\SD}^s}^2 \mkern-5mu -\mu_j\left\|u_j^{\mathbb{H}} \right\|_{L^2}^2 +\mu_j \left\|u_j^{\mathbb{P}} \right\|_{L^2}^2\\
	&\leq\left( 1-\frac{\mu_j}{\Lambda_{k-1}}\right)\left\| u_j^{\mathbb{H}}\right\|_{H_{\SD}^s}^2 \mkern-10mu + \left(\frac{\mu_j}{\Lambda_{k+m}}-1\right)\left\| u_j^{\mathbb{P}}\right\|_{H_{\SD}^s}^2 \\
	&\leq -\frac{\tilde\sigma}{\Lambda_{k+m}}\left\| u_j\right\|_{H_{\SD}^s}^2,
\end{split}
\end{equation}	
where we used the fact that $u_j^{\mathbb{H}}\perp u_j^\mathbb{P}$. On the other hand, by the compactness of the embedding $H_{\Sigma_{\mathcal{D}}}^s(\Omega)\hookrightarrow L^{r}(\Omega)$, $1\leq r< 2_s^*$, we obtain
\begin{align*}
	\left| \int_\Omega f(x, u_j(x))(u_j^{\mathbb{H}}(x)-u_j^{\mathbb{P}}(x))dx\right| & \leq \left\| f(\cdot,u_j(\cdot))\right\|_{L^\frac{q}{q-1}(\Omega)}\left\|u_j^{\mathbb{H}}-u_j^{\mathbb{P}} \right\|_{L^q(\Omega)}\\
	&\leq c \left\| f(\cdot,u_j(\cdot))\right\|_{L^\frac{q}{q-1}(\Omega)}\left\|u_j \right\|_{H_{\SD}^s(\Omega)},
\end{align*}
for a suitable constant $c>0$. Therefore, 
being $u_j\neq 0$, by \eqref{deltalambdakm}, $(f_1)$ and $(f_4)$ we obtain
\begin{equation}\label{deltaclambda}
	\begin{split}
	\frac{\tilde\sigma}{c\Lambda_{k+m}}\left\| u_j\right\|_{H_{\SD}^s(\Omega)} &\leq \left\| f(\cdot,u_j(\cdot))\right\|_{L^\frac{q}{q-1}(\Omega)}\leq \left( \int_\Omega a_1\left( 1+ |u_j(x)|^{q-1}\right)^{\frac{q}{q-1}} dx \right)^{\frac{q-1}{q}}\\
	& \leq \left( c_1 + c_2 \int_\Omega |u_j(x)|^q dx \right)^{\frac{q-1}{q}}\\
	& \leq \left( c_3 + c_4 \int_\Omega F(x,u_j(x)) dx \right)^{\frac{q-1}{q}}
	\end{split}
\end{equation}
for some positive constants $c_i$, $i=1,\ldots,4$. Taking \eqref{intFxUj} into account we conclude that $\{u_j\}$ is bounded in $H_{\SD}^s(\Omega)$. Therefore, up to a subsequence, $u_j\rightharpoonup u_0\in \mathbb{H}_{k-1}\oplus \mathbb{P}_{k+m-1}$ and, by compact embedding, $u_j\to u_0$ in $L^q(\Omega)$ and $u_j\to u_0$ a.e. in $\Omega$, so that, by $(f_1)$ and $(f_3)$ and the Dominated Convergence Theorem, we get
\begin{equation*}
\int_\Omega F(x,u_j(x))dx \to \int_\Omega F(x,u_0(x))dx \quad\text{as } j\to\infty,
\end{equation*}
and
\begin{equation}\label{lst}
\int_\Omega |f(x,u_j(x)|^{\frac{q}{q-1}}dx \to \int_\Omega |f(x,u_0(x)|^{\frac{q}{q-1}}dx \quad\text{as } j\to\infty,
\end{equation}
that, by $(f_4)$ and \eqref{intFxUj}, implies
\begin{equation}\label{limit}
u_0\equiv 0.
\end{equation}
Let us assume that $u_j\to 0$ strongly in $H_{\SD}^s(\Omega)$. First, let us observe that, for any $\varepsilon>0$ and for suitable constants $C,C_\varepsilon>0$ we have
\begin{align*}
\left\| f(\cdot,u_j(\cdot))\right\|_{L^\frac{q}{q-1}(\Omega)} &\leq \left(\int_\Omega\left(  2\varepsilon|u_j(x)|+qC_\varepsilon|u_j(x)|^{q-1}\right)^{\frac{q}{q-1}}dx\right)^{(q-1)/q}, \\
& \leq \left( 2^{\frac{1}{q-1}}\left( (2\varepsilon)^{\frac{q}{q-1}}\left\|u_j \right\|_{L^{\frac{q}{q-1}}(\Omega)}^{\frac{q}{q-1}} + (qC_\varepsilon)^{\frac{q}{q-1}}\left\|u_j \right\|_{L^{q}(\Omega)}^{q} \right) \right)^{\frac{q-1}{q}} \\
& \leq 2\varepsilon \left\| u_j\right\|_{L^{\frac{q}{q-1}}(\Omega)} +q C_\varepsilon \left\| u_j\right\|_{L^q(\Omega)}^{q-1}.
\end{align*}
This, jointly with \eqref{deltaclambda}, leads to the contradiction $\displaystyle 0<\frac{\tilde\sigma}{c\Lambda_{k+m}}\leq 2C\varepsilon$, due to the arbitrariness of $\varepsilon$. If, on the contrary, $u_j\nrightarrow 0$, then $\left\| u_j\right\|_{H_{\SD}^s(\Omega)}\geq C>0$ for some $C>0$ and $j$ large enough. Therefore, by \eqref{deltaclambda}, \eqref{lst}, \eqref{limit} we again get a contradiction, namely $\displaystyle \frac{\tilde\sigma}{c\Lambda_{k+m}}\leq0$.
\end{proof}

\begin{lemma}\label{Ujbounded}
	Let $k\in\N$, $k\geq 2$, let  $\Lambda_k$ be an eigenvalue of multiplicity $m\in\N$, $\lambda\in\R$ and $\Pi_1\vcentcolon=\Pi_{\text{\rm span}\{\varphi_k,\ldots,\varphi_{k+m-1}\}}$ and $\Pi_2\vcentcolon=\Pi_{\mathbb H_{k-1}\oplus \mathbb{P}_{k+m-1}}$. Then, any sequence $\{u_j\}\subset H_{\SD}^s(\Omega)$ such that
	\begin{itemize}
		\item[$(i)$] $\left\{I(u_j)\right\}$ is bounded,
		\item[$(ii)$] $\Pi_1 [u_j]\to 0$ as $j\to\infty$ in $H_{\SD}^s(\Omega)$,
		\item[$(iii)$] $\Pi_2[\nabla I_\lambda(u_j)]\to 0$ as $j\to\infty$ in $H_{\SD}^s(\Omega)$,
	\end{itemize}
is bounded in $H_{\SD}^s(\Omega)$.
\end{lemma}

\begin{proof}
Arguing by contradiction and without losing generality, let $u_j=\Pi_1[u_j]+\Pi_2[u_j]$ be such that $\left\| u_j\right\|_{H_{\SD}^s(\Omega)}\to +\infty $ as $j\to\infty$ and let $u_0\in H_{\SD}^s(\Omega)$ be such that
\begin{equation}\label{convergencesUj}
\frac{u_j}{\left\| u_j\right\|_{H_{\SD}^s(\Omega)}}\rightharpoonup u_0 \text { in } H_{\SD}^s(\Omega), \quad \frac{u_j}{\left\| u_j\right\|_{H_{\SD}^s(\Omega)}}\to u_0 \text { in } L^p(\Omega), \ p\in[1,2^*_s).
\end{equation}
Now, being $H_{\SD}^s(\Omega)$ a Hilbert space and $I_\lambda$ a $C^1$ functional, we define the gradient $\nabla I_\lambda$ of $I_\lambda$, as usual, by
\begin{equation*}
\langle \nabla I_\lambda(u),v \rangle_{H_{\SD}^s}\vcentcolon=\langle I_\lambda'(u),v \rangle_{H^{-s}},
\end{equation*}
for any $u,v\in H_{\SD}^s(\Omega)$. Let $r\in[1,2_s^*]$ and $\mathcal{K}: L^{\frac{r}{r-1}}(\Omega)\to H_{\SD}^s(\Omega)$ be the operator defined by $\mathcal{K}=(-\Delta)^{-s}$, that is, $\mathcal{K}(g)=v$, where $v\in H_{\SD}^s(\Omega)$ is a weak solution to
\begin{equation}
\left\{
	\begin{array}{rl}
		(-\Delta)^{s} v = g  & \text{in } \Omega\\
		               B(v) = 0  & \text{on } \partial\Omega.
	\end{array}
	\right.
\end{equation}
Let us note that,
\begin{align*}
	\left\langle \mathcal K\left( \lambda u+f(\cdot,u)\right),\psi\right\rangle_{H_{\SD}^s(\Omega)} & =\int_\Omega \left(\lambda u(x)+f(x,u(x))\right)\psi(x) dx\\
	& =-\left\langle I'_\lambda(u),\psi\right\rangle_{H^{-s}} + \left\langle u,\psi\right\rangle_{H_{\SD}^s} \\
	&=   -\left\langle \nabla I_\lambda(u),\psi\right\rangle_{H_{\SD}^s} + \left\langle u,\psi\right\rangle_{H_{\SD}^s},
\end{align*}
and, hence
\begin{equation}\label{graduK}
\nabla I_\lambda(u)=u-\mathcal K(\lambda u+f(x,u)),
\end{equation}
for all $u\in H_{\SD}^s(\Omega)$. As a consequence, we find
\begin{align*}
	\left\langle \Pi_2\left[\nabla I_\lambda(u_j)\right],u_j\right\rangle_{H_{\SD}^s}  & = \left\langle \nabla I_\lambda(u_j),u_j\right\rangle_{H_{\SD}^s} - \left\langle \Pi_1\left[\nabla I_\lambda(u_j)\right],u_j\right\rangle_{H_{\SD}^s}\\
	&=\left\| u_j\right\|_{H_{\SD}^s(\Omega)}^2 - \lambda \left\| u\right\|_{L^2(\Omega)}^2 -\int_\Omega f(x,u_j(x))u_j(x) dx \\
	&\ \ - \left\langle \Pi_1\left[u_j-\mathcal K(\lambda u_j+f(x,u_j))\right],u_j\right\rangle_{H_{\SD}^s}.
\end{align*}
Moreover, since $\left\langle \Pi_1[u],v\right\rangle_{H_{\SD}^s}=\left\langle u,\Pi_1[v]\right\rangle_{H_{\SD}^s}$ for all $u,v\in H_{\SD}^s(\Omega)$, one has
\begin{align*}
\left\langle \Pi_1[u_j-\mathcal K(\lambda u_j+f(x,u_j))],u_j\right\rangle_{H_{\SD}^s} &= \left\| \Pi_1u_j\right\|_{H_{\SD}^s}^2 -\lambda \left\langle \Pi_1[u_j],\mathcal K(u_j) \right\rangle_{H_{\SD}^s}\\
&\ \ - \left\langle \Pi_1[u_j],\mathcal K(f(x,u_j))\right\rangle_{H_{\SD}^s},
\end{align*}
and, by \eqref{graduK},
\begin{align*}
	\lambda \left\langle \Pi_1[u_j],\mathcal{K}(u_j)\right\rangle_{H_{\SD}^s}&+ \left\langle \Pi_1[u_j],\mathcal{K}(f(x,u_j))\right\rangle_{H_{\SD}^s}\\
	& = \lambda \left\| \Pi_1[u_j]\right\|_{L^2(\Omega)}^2 + \int_\Omega f(x,u_j(x))\Pi_1[u_j](x) dx.
\end{align*}
Therefore, we deduce
\begin{equation}\label{relationQnabla}
	\begin{split}
		\left\langle \Pi_2\left[\nabla I_\lambda(u_j)\right],u_j\right\rangle_{H_{\SD}^s} & = 2I_\lambda(u_j) +2\int_\Omega F(x, u_j(x)) dx\\
		&\ \  -\int_\Omega f(x, u_j(x))\Big(u_j(x)-\Pi_1[u_j](x)\Big) dx\\
		&\ \  - \left\| \Pi_1[u_j]\right\|_{H_{\SD}^s}^2 +\lambda\left\| \Pi_1[u_j]\right\|_{L^2(\Omega)}^2.
	\end{split}
\end{equation}
Thus, by $(i)-(iii)$ it follows that
\begin{equation}\label{fractiongoingtozero}
\frac{\displaystyle 2\int_\Omega F(x, u_j(x)) dx -\int_\Omega f(x, u_j(x))u_j(x) dx +\int_\Omega f(x, u_j(x)) \Pi_1[u_j](x) dx}{\left\| u_j\right\|^q}\to 0
\end{equation}
as $j\to \infty$.  Next, by $(f_1)$ and Proposition \ref{phikbounded} we obtain
\begin{equation*}
|f(x,u_j(x))\Pi_1[u_j](x)| \leq a_1\left\| \Pi_1[u_j]\right\|_{L^\infty(\Omega)}(1+|u_j(x)|^{q-1}),
\end{equation*}
for a.e. $x\in\Omega$, and hence, since $(ii)$ forces $\left\| \Pi_1[u_j]\right\|_{L^\infty(\Omega)}\to 0$ (on $\text{span}\{\varphi_k,\ldots,\varphi_{k+m-1}\}$ all norms are equivalent), by Dominated Convergence
\begin{equation*}
\frac{\displaystyle \int_\Omega f(x, u_j(x)) \Pi_1[u_j](x) dx}{\left\| u_j\right\|^q}\to 0,
\end{equation*}
as $j\to\infty$.
Thus, by $(f_4)$ and \eqref{fractiongoingtozero} we find
\begin{equation*}
0\geq\frac{(2-q)\displaystyle\int_\Omega F(x, u_j(x)) dx}{\|u_j\|_{H_{\SD}^s(\Omega)}}\geq\frac{\displaystyle 2\int_\Omega F(x, u_j(x)) dx -\int_\Omega f(x, u_j(x))u_j(x) dx}{\left\| u_j\right\|_{H_{\SD}^s(\Omega)}^q}\to 0,
\end{equation*}
and, therefore,
\begin{equation*}
\lim_{j\to\infty}\frac{\displaystyle \int_\Omega F(x,u_j(x))dx}{\left\| u_j\right\|_{H_{\SD}^s(\Omega)}^q} = \lim_{j\to\infty}\frac{\displaystyle \left\|u_j\right\|_{L^q(\Omega)}^q}{\left\| u_j\right\|_{H_{\SD}^s(\Omega)}^q}=0,
\end{equation*}
where we have used the assumption $\left\| u_j\right\|_{H_{\SD}^s(\Omega)}\to +\infty $ as $j\to\infty$.
The above limit, together with \eqref{convergencesUj}, forces $u_0\equiv 0$.

Now, observe that one has
\begin{equation*}
\frac{I_\lambda(u_j)}{\left\|u_j \right\|_{H_{\SD}^s(\Omega)}^2}=\frac12-\frac{\lambda \left\|u_j \right\|_{L^2(\Omega)}^2}{2 \left\|u_j \right\|_{H_{\SD}^s(\Omega)}^2}-\frac{\displaystyle \int_\Omega  F(x,u_j(x))dx}{\left\|u_j \right\|_{H_{\SD}^s(\Omega)}^2}\to0.
\end{equation*}
Then, because of \eqref{convergencesUj}, used here with $p=2$, and the fact that $u_0\equiv0$, we find
\begin{equation}\label{eqa}
\frac{\displaystyle \int_\Omega  F(x,u_j(x))dx}{\left\|u_j \right\|_{H_{\SD}^s(\Omega)}^2}\to\frac12,
\end{equation}
which, together with $(f_2)$ and $\left\| u_j\right\|_{H_{\SD}^s(\Omega)}\to +\infty $ as $j\to\infty$, implies
\begin{equation}\label{eqb}
\left\| u_j\right\|_{L^q(\Omega)}^q\leq C \|u_j\|_{H_{\SD}^s(\Omega)}^2,
\end{equation}
for any $j\in\mathbb{N}$ and for some constant $C>0$. To continue, by $(f_1)$ and H\"older's inequality, we get
\begin{align*}
	\int_\Omega \left| f(x,u_j(x))\Pi_1[u_j]\right| dx & \leq a_1\left\|\Pi_1[u_j] \right\|_{L^\infty(\Omega)} \left( |\Omega|+\left\| u_j\right\|_{L^{q-1}(\Omega)}^{q-1}\right)\\
	& \leq c_1\left\|\Pi_1[u_j] \right\|_{L^\infty(\Omega)} (1 + \left\|u_j \right\|_{L^q(\Omega)}^{q-1}),
\end{align*}
for some $c_1>0$, and ought to \eqref{eqb}, we obtain
\begin{equation*}
\frac{\displaystyle\int_\Omega \left| f(x,u_j(x))\Pi_1[u_j](x)\right|  dx}{\|u_j\|_{H_{\SD}^s(\Omega)}^2 } \leq c_2\left\|\Pi_1[u_j] \right\|_{L^\infty(\Omega)} \left( \frac{1}{\|u_j\|_{H_{\SD}^s(\Omega)}^2}+\frac{1}{\|u_j\|_{H_{\SD}^s(\Omega)}^{2/q}}\right)\to 0.
\end{equation*}
As a result, dividing both sides of \eqref{relationQnabla} by $\|u_j\|_{H_{\SD}^s(\Omega)}^2$ and taking also account of  $(i)-(iii)$, we conclude that
\begin{equation*}
\frac{\displaystyle 2\int_\Omega  F(x,u_j(x))dx -\int_\Omega f(x,u_j(x)) dx}{\left\| u_j\right\|_{H_{\SD}^s(\Omega)}^2}\to 0.
\end{equation*}
By using $(f_4)$, this yields
\begin{equation*}
\frac{\displaystyle 2\int_\Omega  F(x,u_j(x))dx}{\left\| u_j\right\|_{H_{\SD}^s(\Omega)}^2}\to 0,
\end{equation*}
in contradiction with \eqref{eqa}.
\end{proof}	

We are now in a position to show that $I_\lambda$ satisfies the $\nabla$-condition.

\begin{proposition}\label{nablacond}
Let $k\in\N$, $k\geq 2$, and let $\Lambda_k$ be an eigenvalue of multiplicity $m\in\N$. Then, for any $\sigma>0$ there exists $\varepsilon_\sigma>0$ such that, for any $\lambda\in[\Lambda_{k-1}+\sigma, \Lambda_{k+m}-\sigma]$ and for any $\varepsilon',\varepsilon''\in(0,\varepsilon_\sigma)$, with $\varepsilon'<\varepsilon''$, the functional $I_\lambda$ satisfies  $(\nabla)(I_\lambda,\mathbb{H}_{k-1}\oplus \mathbb{P}_{k+m-1},\varepsilon',\varepsilon'')$.
\end{proposition}

\begin{proof}
We argue by contradiction. Let us assume the existence of $\sigma>0$ such that for every $\varepsilon_0>0$ there exists $\bar\lambda\in[\Lambda_{k-1}+\sigma,\Lambda_{k+m}-\sigma]$ and $\varepsilon'<\varepsilon''$ in $(0,\varepsilon_0)$ such that $(\nabla)(I_{\bar\lambda},\mathbb{H}_{k-1}\oplus \mathbb{P}_{k+m-1},\varepsilon',\varepsilon'')$ does not hold.

Take $\varepsilon_0>0$ as in Lemma \ref{uniquecritpoint}, i.e., such that $u=0$ is the only critical point of $I_{\bar\lambda}$ on $\mathbb{H}_{k-1}\oplus \mathbb{P}_{k+m-1}\cap I_{\bar\lambda}^{-1}([-\varepsilon_0,\varepsilon_0])$, and let $\{u_j\}\subset H_{\SD}^s(\Omega)$ be such that
$$
I_{\bar\lambda}(u_j)\in[\varepsilon',\varepsilon''], \text{ for all } j\in\N, \quad \text{dist}(u_j,\mathbb{H}_{k-1}\oplus \mathbb{P}_{k+m-1})\to 0,
$$
and
\begin{equation}\label{Pnabla}
	\begin{split}
\Pi_{\mathbb{H}_{k-1}\oplus P_{k+m-1}}\nabla I_{\bar\lambda}(u_j) &= u_j-\Pi_{\text{span}\{\varphi_k,\ldots,\varphi_{k+m-1}\}}u_j\\
&\ \ -\Pi_{\mathbb{H}_{k-1}\oplus \mathbb{P}_{k+m-1}}\mathcal{K}(\bar\lambda u_j +f(x,u_j))\to 0,
\end{split}
\end{equation}
as $j\to\infty$. By Lemma \ref{Ujbounded}, $\{u_j\}$ is bounded in $H_{\SD}^s(\Omega)$ so, up to a subsequence, $u_j\rightharpoonup u_0$ in $H_{\SD}^s(\Omega)$ and $u_j\to L^r(\Omega)$, for all $r\in[1,2^*_s)$ . Since $\mathcal K=(-\Delta)^{-s}$ is a compact operator, by standard dominated convergence arguments we get
$$
\Pi_{\mathbb{H}_{k-1}\oplus \mathbb{P}_{k+m-1}}\mathcal K(\bar\lambda u_j +f(x,u_j))\to \Pi_{\mathbb{H}_{k-1}\oplus \mathbb{P}_{k+m-1}}\mathcal K(\bar\lambda u_0 +f(x,u_0))
$$
as $j\to \infty$. As a result, taking the limit as $j\to\infty$ in \eqref{Pnabla} we obtain
$$
u_j\to \Pi_{\mathbb{H}_{k-1}\oplus \mathbb{P}_{k+m-1}}(-\Delta)^{-s}(\bar\lambda u_0 +f(x,u_0))=u_0 \text { in } H_{\SD}^s(\Omega).
$$
In addition, $u_0$ turns out to be a critical point of $I_{\bar\lambda}|_{\mathbb{H}_{k-1}\oplus \mathbb{P}_{k+m-1}}$. Indeed, by \eqref{Pnabla}, we get
\begin{align*}
	\left\langle I'_{\bar\lambda}(u_j),\psi\right\rangle_{H_{\SD}^s} & = \left\langle u_j,\psi\right\rangle_{H_{\SD}^s} -\lambda\int_\Omega u_j(x)\psi(x) dx -\int_\Omega f(x,u_j(x))\psi(x)dx \to 0
\end{align*}
as $j\to\infty$, for all $\psi\in \mathbb{H}_{k-1}\oplus \mathbb{P}_{k+m-1}$, and by using again the dominated convergence theorem,
$$
\left\langle I'_{\bar\lambda}(u_0),\psi\right\rangle_{H_{\SD}^s} = \left\langle u_0,\psi\right\rangle_{H_{\SD}^s} -\lambda\int_\Omega u_0(x)\psi(x) dx -\int_\Omega f(x,u_0(x))\psi(x)dx.
$$
So, it must be $u_0=0$. Since $I_{\bar\lambda}(u_j)\geq\varepsilon'>0$, we get $I_{\bar\lambda}(u_0)>0$, a contradiction.
\end{proof}

Finally we show the validity of Palais-Smale condition for $I_\lambda$ at any level.
\begin{proposition}\label{propPS}
The functional $I_\lambda$ satisfies $(PS)_c$ for any $\lambda\in(0,+\infty)$ and any $c\in\R$.
\end{proposition}

\begin{proof}
	Let $\lambda>0$, $c\in\R$ and  let $\{u_j\}\subset H_{\SD}^s(\Omega)$ satisfy
		\begin{equation}\label{PSsequence}
			I_\lambda(u_j)\to c \quad\text{and } \quad I'_\lambda(u_j)\to 0 \quad \text{as } j\to\infty.
		\end{equation}
	As usual, we start proving that $\{u_j\}$ is bounded in $H_{\SD}^s(\Omega)$. First, we note that
	\begin{equation*}
		\left|I_\lambda(u_j)\right|\leq c_1\quad\text{\and}\quad \left|\left\langle I_\lambda'(u_j),\frac{u_j}{\|u_j\|_{H_{\SD}^s(\Omega)}}\right\rangle_{H^{-s}} \right| \leq c_2,
	\end{equation*}
	for some $c_1,c_2>0$.  On the other hand, because of $(f_2)$ and $(f_3)$ and $\varepsilon$-Young's inequality
	$$
	t^2\leq\varepsilon t^q+\varepsilon^{-\frac{2}{q-2}}\left(\frac{2}{q}\right)^{\frac{2}{2-q}}\frac{q-2}{q}, \quad t>0,\; \varepsilon>0,
	$$
	we get
	\begin{align*}
		c_1+c_2\left\| u_j\right\|_{H_{\SD}^s(\Omega)} &\geq I_\lambda(u_j)-\frac{1}{\mu} \left\langle  I_\lambda'(u_j), u_j\right\rangle_{H^{-s}}\\
		& = \left(\frac{1}{2}-\frac{1}{\mu}\right)\left(\left\| u_j\right\|_{H_{\SD}^s(\Omega)}^2 - \lambda\left\|u_j \right\|_{L^2(\Omega)}^2\right)-\int_\Omega F(x,u_j)-\frac{1}{\mu}f(x,u_j)u_j\,dx\\
		& \geq \left(\frac{1}{2}-\frac{1}{\mu}\right)\left(\left\| u_j\right\|_{H_{\SD}^s(\Omega)}^2 - \lambda\left\|u_j \right\|_{L^2(\Omega)}^2\right) +\left(\frac{q}{\mu}-1 \right) \int_\Omega F(x,u_j) dx\\
		&\geq \left(\frac{1}{2}-\frac{1}{\mu}\right)\left(\left\| u_j\right\|_{H_{\SD}^s(\Omega)}^2 - \lambda\left\|u_j \right\|_{L^2(\Omega)}^2\right)\\
		&\ \ + a_2 \left(\frac{q}{\mu}-1 \right)\left\| u_j\right\|_{L^q(\Omega)}^q -a_3\left(\frac{q}{\mu}-1 \right)|\Omega|\\
		& \geq \left(\frac{1}{2}-\frac{1}{\mu}\right)\left\| u_j\right\|_{H_{\SD}^s(\Omega)}^2 +\left( a_2\left(\frac{q}{\mu}-1\right)-\lambda\left(\frac{1}{2}-\frac{1}{\mu}\right)\varepsilon \right) \left\| u_j\right\|_{L^q(\Omega)}^q\\
		&\ \ -\lambda\left(\frac{1}{2}-\frac{1}{\mu}\right)\left(\frac{2}{q}\right)^{\frac{2}{q-2}}\frac{q-2}{q}|\Omega|\varepsilon^{-\frac{2}{q-2}} -a_3\left(\frac{q}{\mu}-1 \right)|\Omega|.
	\end{align*}
	Taking $\varepsilon\in\left(0,\frac{2a_2(q-\mu)}{\lambda(\mu-2)} \right)$ we conclude
	$$
	c_1+c_2\left\| u_j\right\|_{H_{\SD}^s(\Omega)}\geq \left(\frac{1}{2}-\frac{1}{\mu}\right)\left\| u_j\right\|_{H_{\SD}^s(\Omega)}^2 -C_\varepsilon -c_3,
	$$
	for some $c_3>0$, so $\{u_j\}$ is bounded in $H_{\SD}^s(\Omega)$. Then, up to a subsequence, $u_j\rightharpoonup u_0\in H_{\SD}^s(\Omega)$, i.e.
	\begin{equation}\label{relatscalarprod}
		\left\langle u_j,\psi \right\rangle_{H_{\SD}^s(\Omega)}\to\left\langle u_0,\psi\right\rangle_{H_{\SD}^s(\Omega)}, \quad \text{for all } \psi\in H_{\SD}^s(\Omega).
	\end{equation}
	Next, let us note that
	\begin{equation}\label{strong}
		\begin{split}
			\|u_j-u_0\|_{H_{\Sigma_{\mathcal{D}}}^s(\Omega)}^2&=\|u_j\|_{H_{\Sigma_{\mathcal{D}}}^s(\Omega)}^2+\|u_0\|_{H_{\Sigma_{\mathcal{D}}}^s(\Omega)}^2-2\langle u_j,u_0\rangle_{H_{\Sigma_{\mathcal{D}}}^s}.
		\end{split}
	\end{equation}
	By the compactness of the embedding $H_{\Sigma_{\mathcal{D}}}^s(\Omega)\hookrightarrow L^{r}(\Omega)$, $1\leq r< 2_s^*$,
	up to a subsequence, we then obtain
	\begin{equation}\label{eq1}
		\begin{split}
			u_j&\to u_0\quad\text{in }L^r(\Omega),\ \text{for any } 1\leq r< 2_s^*,\\
			u_j&\to u_0\quad\text{a.e. in } \Omega,
		\end{split}
	\end{equation}
	for $j\to+\infty$. By the continuity in $t\in\R$ of the map $t\mapsto f(\cdot,t)$ and the Dominated Convergence Theorem, it follows that
	\begin{equation}\label{eq2}
		\int_\Omega f(x,u_j(x))u_j(x)dx \to \int_\Omega f(x,u_0(x))u_0(x)dx 
	\end{equation}
	and
	\begin{equation}\label{eq3}
		\int_\Omega f(x,u_j(x))u_0(x)dx \to \int_\Omega f(x,u_0(x))u_0(x)dx 
	\end{equation}
	as $j\to \infty$. Moreover, since by \eqref{PSsequence}
	\begin{equation*}
		\begin{split}
			\langle I_\lambda'(u_j),u_j\rangle_{H^{-s}}&=\int_{\Omega}|(-\Delta)^{\frac{s}{2}}u_j|^{2}dx-\lambda\int_{\Omega}|u_j(x)|^2dx-\int_{\Omega}f(x,u_j(x))u_j(x)dx\\
			&= \left\langle u_j,u_j \right\rangle_{H_{\SD}^s(\Omega)}-\lambda\int_{\Omega}|u_j(x)|^2dx-\int_{\Omega}f(x,u_j(x))u_j(x)dx\to 0,
		\end{split}
	\end{equation*}
	thus, by \eqref{eq1} and \eqref{eq2} we find
	\begin{equation}\label{eq4}
		\left\langle u_j,u_j \right\rangle_{H_{\SD}^s(\Omega)}\to\lambda\int_{\Omega}|u_0(x)|^2dx+\int_{\Omega}f(x,u_0(x))u_0(x)dx \quad\text{as } j\to \infty.
	\end{equation}
	Combining \eqref{PSsequence} and \eqref{relatscalarprod}, tested with $\psi=u_0$, we also get
	\begin{equation}\label{eq5}
		\langle u_0,u_0 \rangle_{H_{\SD}^s(\Omega)}=\lambda\int_{\Omega}|u_0(x)|^2dx+\int_{\Omega}f(x,u_0(x))u_0(x)dx.
	\end{equation}
	Therefore, we conclude
	\begin{equation*}
		\left\|  u_j\right\|_{H_{\SD}^s(\Omega)}^2 \to \| u_0\|_{H_{\SD}^s(\Omega)}^2, \quad\text{as }j\to \infty.
	\end{equation*}
	Then, by \eqref{strong}
	\begin{equation*}
		\begin{split}
			\|u_j-u_0\|_{H_{\Sigma_{\mathcal{D}}}^s(\Omega)}^2&=\|u_j\|_{H_{\Sigma_{\mathcal{D}}}^s(\Omega)}^2+\|u_0\|_{H_{\Sigma_{\mathcal{D}}}^s(\Omega)}^2-2\langle u_j,u_0\rangle_{H_{\Sigma_{\mathcal{D}}}^s}\\
			& \to 2\| u_0\|_{H_{\SD}^s(\Omega)}^2-2\langle u_0,u_0 \rangle_{H_{\SD}^s(\Omega)}=0,
		\end{split}
	\end{equation*}
and the proof is concluded.
\end{proof}

\section{Proof of Theorem \ref{mainresult}}\label{proofofmainresult}
We premise this result related to the behavior of $I_\lambda$ on $\mathbb H_{k+m-1}$ as $\lambda$ approaches an eigenvalue $\Lambda_k$ of multiplicity $m$.

\begin{lemma}\label{lemaa}
	Let $k\in\N$, $k\geq 2$, and let $\Lambda_k$ be an eigenvalue of multiplicity $m\in\N$. Then
	$$
	\lim_{\lambda\to\Lambda_k}\sup_{u\in \mathbb{H}_{k+m-1}}I_\lambda(u)=0.
	$$
\end{lemma}

\begin{proof}
	By $(f_2)$, $I_\lambda$ attains a maximum on $\mathbb{H}_{k+m-1}$. By contradiction, assume that there exists a sequence $\{\mu_j\}$ such that
	\begin{equation}\label{eqc}
		\mu_j\to\Lambda_k, \quad \text{as } j\to\infty,
	\end{equation}
	and $\{u_j\}\subset \mathbb{H}_{k+m-1}$ and $\varepsilon>0$ such that
	\begin{equation}\label{eqd}
		I_{\mu_j}(u_j)=\sup\limits_{u\in \mathbb{H}_{k+m-1}}I_{\mu_j}(u)\geq\varepsilon.
	\end{equation}
	for any $j\in\mathbb{N}$. We have two options, according to whether $\{u_j\}$ is unbounded in $H_{\SD}^s(\Omega)$ or not. In the first case, without loss of generality we can take $\left\| u_j\right\|_{H_{\SD}^s(\Omega)}\to +\infty$; because of \eqref{eqd} and $(f_2)$, we get
	\begin{equation*}
		0<\varepsilon\leq I_{\mu_j}(u_j)\leq\frac{1}{2}\left\| u_0\right\|_{H_{\SD}^s(\Omega)}^2-\frac{\mu_j}{2}\int_{\Omega}|u_0(x)|^2dx-a_2\int_{\Omega}|u_j(x)|^qdx+a_3|\Omega|\to-\infty,
	\end{equation*}
	as $j\to\infty$, since $q>2$ and all norms on $H_{k+m-1}$ are equivalent. If, on the contrary,  $\{u_j\}$ is bounded, by \eqref{pre} and \eqref{eqc} we get
	\begin{equation*}
		I_{\mu_j}(u_j)\to I_{\Lambda_k}(u_0),
	\end{equation*}
	 and by $(f_4)$, Lemma \ref{varcharacteigen} and \eqref{eqd}, we find
	\begin{equation*}
		\begin{split}
			\varepsilon\leq I_{\Lambda_k}(u_0)&=\frac{1}{2}\left\| u_0\right\|_{H_{\SD}^s(\Omega)}^2-\frac{\Lambda_k}{2}\int_{\Omega}|u_0(x)|^2dx-\int_{\Omega}F(x,u_0(x))dx\\
			&\leq\frac12(\Lambda_{k+m-1}-\Lambda_k)\int_{\Omega}|u_0(x)|^2dx-\int_{\Omega}F(x,u_0(x))dx\\
			&\leq 0.
		\end{split}
	\end{equation*}
Being both cases contradictory, we get the conclusion.
\end{proof}


\begin{proof}[Proof of Theorem \ref{mainresult}]
Let us first show that, if $\Lambda_k$ has multiplicity $m\in\N$, 
then there exists $\delta>0$ such that for all $\lambda\in(\Lambda_{k}-\delta,\Lambda_{k})$, problem \eqref{problem} has two non-trivial solutions $u_i$ satisfying
\begin{equation}\label{bound}
	0<I_{\lambda}(u_i)\leq\sup\limits_{u\in \mathbb{H}_{k+m-1}}I_\lambda(u),
\end{equation}
for $i=1,2$. To this end, fix $\sigma>0$ and let $\varepsilon_\sigma>0$ be as in Proposition \ref{nablacond}. Then, for all $\lambda\in[\Lambda_{k-1}+\sigma,\Lambda_{k+m}-\sigma]$ and for every $\varepsilon',\varepsilon''\in(0,\varepsilon_\sigma)$, condition $(\nabla)(I_\lambda,\mathbb{H}_{k-1}\oplus \mathbb{P}_{k+m-1},\varepsilon',\varepsilon'')$ holds. Then, on account of Lemma \ref{lemaa}, we can find $\delta\leq\sigma$ such that
\begin{equation}\label{eqe}
\sup\limits_{u\in\mathbb{H}_{k+m-1}}I_\lambda(u)<\varepsilon''
\end{equation}
if $\lambda\in (\Lambda_k-\delta,\Lambda_k)$. Then, taking also account of Propositions \ref{prop_supinf} and \ref{propPS}, by Theorem \ref{nablathm}, $I_\lambda$ has two critical points $u_1$ and $u_2$ such that $I_\lambda(u_i)\in [\varepsilon',\varepsilon'']$. In particular, both $u_1$ and $u_2$ are two non-trivial solutions to \eqref{problem} satisfying \eqref{bound} since $\varepsilon''>0$ is arbitrary.
For the existence of a third solution to \eqref{problem} when
$\lambda\in(\Lambda_{k-1},\Lambda_k)$ we show that $I_\lambda$ has an additional linking structure and so we can appeal to the classical \cite[Theorem 5.3]{rab1986minimax}.
Indeed, considering the decomposition $H_{\SD}^s(\Omega)=\mathbb{H}_{k-1}\oplus\mathbb{P}_{k-1}$, by retracing the proof of Proposition \ref{prop_supinf}, there exist $\varrho,\beta>0$ such that has
\begin{equation*}
I_\lambda(u)\geq\beta,\qquad\text{for all }u\in\partial B_\varrho(\mathbb{P}_{k-1}).
\end{equation*}
On the other hand, by \eqref{diseqinHk-1}, we have ${I_\lambda}_{|\mathbb{H}_{k-1}}\leq 0$. By \eqref{ineqilambda} one has also $I_\lambda(u)\leq 0$ for $u\in\mathbb{H}_{k-1}\oplus\text{span}\{\varphi_k\}=\mathbb{ H}_k$ and $\left\| u\right\|_{H_{\SD}^s(\Omega)}\geq\bar R$, for some $\bar R>\varrho$. Then, for $R>0$ large, setting
$$
Q:=(B_R(H_{\SD}^s(\Omega))\cap\mathbb{H}_{k-1})\oplus\{r\varphi_k,\ r\in(0,R)\},
$$
one has ${I_\lambda}_{|\partial Q}\leq 0$ (here $\partial Q$ refers to the boundary of $Q$ relative to $\mathbb{H}_{k-1}\oplus\text{span}\{\varphi_k\}$). So, by \cite[Remark 5.5]{rab1986minimax}, $I_\lambda$ has a third critical point $u_3$ satisfying
\begin{equation*}
I_\lambda(u_3)\geq \inf\limits_{u\in \partial B_\varrho(\mathbb{P}_{k-1})}I_\lambda(u)\geq\beta.
\end{equation*}
Moreover, $u_3\neq u_i$, $i=1,2$, because, on account of \ref{bound} and Lemma \ref{lemaa}, for $\lambda$ close to $\Lambda_k$ we obtain
\begin{equation*}
I_\lambda(u_i)\leq\sup\limits_{u\in\mathbb{H}_{k+m-1}}I_\lambda(u)<\inf\limits_{u\in \partial B_\varrho(\mathbb{P}_{k-1})}I_\lambda(u).
\end{equation*}
The theorem is then completely proved.
\end{proof}



\begin{thebibliography}{99}

\bibitem{mu0} L. Appolloni, G. Molica Bisci and S. Secchi, \textit{Multiple solutions for Schrödinger equations on Riemannian manifolds via $\nabla$-theorems}, Ann. Global Anal. Geom. {\bf 63} (2023), no. 1, 22 pp.

\bibitem{Barrios2012}
B. Barrios, E. Colorado, A. de Pablo and U. S\'{a}nchez, {\em On some critical problems for the fractional Laplacian operator}, J. Differential Equations {\bf 252} (2012), no. 11, pp. 6133--6162.

\bibitem{Barrios2020} B. Barrios and M. Medina, {\it Strong maximum principles for fractional elliptic and parabolic problems with mixed boundary conditions}, Proc. Roy. Soc. Edinburgh Sect. A {\bf 150}, No 1 (2020), 475--495.

\bibitem{bracoldepsan2013a} C. Br\"{a}ndle, E. Colorado, A. de Pablo and U. S\'{a}nchez, A concave-convex elliptic problem involving the fractional Laplacian, Proc. Roy. Soc. Edinburgh, {\bf 143}A (2013), 39--71.
	
\bibitem{Cabre2010} X. Cabr\'{e} and J. Tan, {\em Positive solutions of nonlinear problems involving the square root of the Laplacian}, Adv. Math. {\bf 224} (2010), no. 5, pp. 2052--2093.

\bibitem{cafsil2007an} L. Caffarelli and L. Silvestre, An extension problem related to the fractional Laplacian, Comm. Partial Differential Equations {\bf 32} (2007), no. 7-9, 1245--1260.

\bibitem{Capella2011}
A. Capella, J. D\'{a}vila, L. Dupaigne  and Y. Sire, {\em Regularity of radial extremal solutions for some non-local semilinear equations}, Comm. Partial Differential Equations {\bf 36} (2011), no. 8, pp. 1353--1384.	
	
\bibitem{cafsil2007an} L. Caffarelli and L. Silvestre, \textit{An extension problem related to the fractional Laplacian}, Comm. Partial Differential Equations {\bf 32} (2007), no. 7-9, 1245--1260.

\bibitem{carcolleoort2020regularity} J. Carmona, E. Colorado, T. Leonori and A. Ortega, {\it Regularity of solutions to a fractional elliptic problem with mixed Dirichlet--Neumann boundary data}, Adv. Calc. Var. {\bf 14} (4) (2021), 521--539.

\bibitem{colort2019the} E. Colorado and A. Ortega, \textit{The Brezis-Nirenberg problem for the fractional Laplacian with mixed Dirichlet-Neumann boundary conditions}, J. Math. Anal. Appl. {\bf 473} (2019), no. 2, 1002--1025.

\bibitem{colper2003semilinear} E. Colorado and I. Peral, \textit{Semilinear elliptic problems with mixed Dirichlet-Neumann boundary conditions}, J. Funct. Anal. {\bf 199} (2003), no. 2, 468--507.

\bibitem{valpal}{E. Di Nezza, G. Palatucci and E. Valdinoci}, {\em Hitchhiker's guide to the fractional Sobolev spaces}, Bull. Sci. Math., {\bf 136} (2012), 521--573.

\bibitem{Leonori2018} T. Leonori, M. Medina, I. Peral, A. Primo and F. Soria, {\it Principal eigenvalue of mixed problem for the fractional Laplacian: Moving the boundary conditions} J. Differential Equations {\bf 265} (2) (2018), 593--619.

\bibitem{Lions1972}
J.-L. Lions and E. Magenes, {\em Non-homogeneous boundary value problems and applications. Vol. I}, Springer-Verlag, New York-Heidelberg, 1972.
\newblock Translated from the French by P. Kenneth, Die Grundlehren der mathematischen Wissenschaften, Band 181. Springer-Verlag, New York-Heidelberg, 1972. xvi+357 pp.

\bibitem{liopactri1988best} P.L. Lions, F. Pacella and M. Tricarico, \textit{Best Constants in Sobolev Inequalities for Functions Vanishing on Some Part of the Boundary and Related Questions}, Indiana Univ. Math. J. {\bf 37} (2) (1988), 301--324.

\bibitem{mms} {P. Magrone, D. Mugnai and R. Servadei}, \emph{Multiplicity of solutions for semilinear variational inequalities via linking and $\nabla-$theorems}, J. Differential Equations {\bf 228} (2006), 191--225.

\bibitem{marsac1997some} A. Marino and C. Saccon, \textit{Some variational theorems of mixed type and elliptic problems with jumping nonlinearities}, Ann. Scuola Norm. Sup. Pisa Cl. Sci. {\bf 25} (4) (1997), 631--665.

\bibitem{mu1} D. Mugnai, \textit{Multiplicity of critical points in presence of a linking: application to a superlinear boundary value problem}, NoDEA Nonlinear Differential Equations Appl. {\bf 11} (2004), no. 3, 379--391.

\bibitem{mu4} D. Mugnai, \emph{Four nontrivial solutions for subcritical exponential equations}, Calc. Var. Partial Differential Equations, {\bf 32} (2008), 481--497.

\bibitem{mu2} G. Molica Bisci, D. Mugnai and R. Servadei, \textit{On multiple solutions for nonlocal fractional problems via $\nabla$-theorems}, Differential Integral Equations {\bf 30} (2017), no. 9-10, 641--666.

\bibitem{MRS}{G. Molica Bisci, V. R\u{a}dulescu and R. Servadei},
Variational Methods for Nonlocal Fractional Problems. With a Foreword by Jean Mawhin, {\em Encyclopedia of Mathematics and its Applications}, {\em Cambridge University Press} \textbf{162}, Cambridge, 2016. ISBN
9781107111943.

\bibitem{Ortega2023} A. Ortega, {\it Concave-convex critical problems for the spectral fractional laplacian with mixed boundary conditions}. Fract. Calc. Appl. Anal. {\bf 26}, 305--335 (2023). 

\bibitem{rab1986minimax} P.H. Rabinowitz, \textit{Minimax methods in critical point theory with applications to differential equations}, CBMS Reg. Conf. Ser. Math. {\bf 65}, American Mathematical Society, Providence, RI (1986).

\bibitem{Stroock1985} N.T. Varopoulos. \textit{Hardy-Littlewood theory for semigroups}. J. Funct. Anal., {\bf 63} (1985), 240--260.

\end{thebibliography}
\end{document}